\documentclass[12pt]{amsart}
\usepackage{latexsym,amssymb,amsmath,hyperref,amsthm,amsfonts,
caption,subcaption,tikz,comment}
\usetikzlibrary{
intersections, arrows.meta,
automata,er,calc,
backgrounds,
mindmap,folding,
patterns,
decorations.markings,
fit,decorations,
shapes,matrix,
positioning,
shapes.geometric,
arrows,through, graphs, graphs.standard
}

\numberwithin{equation}{section}
\usetikzlibrary{positioning}
\usetikzlibrary{arrows}

\newtheorem{theorem}{Theorem}[section]
\newtheorem{lemma}[theorem]{Lemma}

\newtheorem{corollary}[theorem]{Corollary}

\theoremstyle{definition}
 
\newtheorem{remark}[theorem]{Remark}

\newcommand{\K}{\mathbb{K}}
\newcommand{\N}{\mathbb{N}}
\newcommand{\Z}{\mathbb{Z}}
\DeclareMathOperator{\im}{im}

\DeclareMathOperator{\reg}{reg}

\DeclareMathOperator{\pdim}{pdim}
\DeclareMathOperator{\depth}{depth}

\DeclareMathOperator{\mat}{mat}

\begin{document}

 
\title{Comparing invariants of toric ideals of 
bipartite graphs}
\thanks{Version: \today}

\author[K. Bhaskara]{Kieran Bhaskara}
\author[A. Van Tuyl]{Adam Van Tuyl}

\address[K. Bhaskara, A. Van Tuyl]
{Department of Mathematics and Statistics\\
McMaster University, Hamilton, ON, L8S 4L8}
\email{ kieran.bhaskara@mcmaster.ca,
vantuyl@math.mcmaster.ca}

\keywords{Toric ideals, bipartite graphs, $h$-polynomial,
regularity, projective dimension, depth, dimension}
\subjclass[2020]{13D02, 13P10, 13D40, 14M25, 05E40}

 
\begin{abstract}
Let $G$ be a finite simple graph and let $I_G$ denote 
its associated toric ideal in the polynomial ring $R$.  
For each integer $n\geq 2$, we completely determine all
the possible values for the tuple 
$({\rm reg}(R/I_G), \deg(h_{R/I_G}(t)),
{\rm pdim}(R/I_G), {\rm depth}(R/I_G),\dim(R/I_G))$ when
$G$ is a connected bipartite graph on $n$ vertices.
\end{abstract}
\maketitle
\section{Introduction}
Let $I$ be a homogeneous ideal of a polynomial 
ring $S = \mathbb{K}[x_1,\ldots,x_m]$ over an 
algebraically closed field $\mathbb{K}$ of 
characteristic zero.  Associated
with $I$ are a number of homological invariants that 
are encoded in the minimal graded free resolution of
$S/I$;  some of these invariants are
the (Castelnuovo--Mumford) regularity $\reg(S/I)$,
the projective dimension $\pdim(S/I)$,
the (Krull) dimension $\dim(S/I)$, the 
depth $\depth(S/I)$
and $\deg(h_{S/I}(t))$,
the degree of the $h$-polynomial of $S/I$. Four of these
invariants are related by the inequality
\[\reg(S/I) - \deg(h_{S/I}(t)) \leq \dim(S/I) - 
\depth(S/I),\]
(see, for example \cite[Corollary B.28]{Vas}), 
while the depth and projective dimension are related via 
the well-known Auslander--Buchsbaum formula.

A recent program in combinatorial commutative algebra 
is to understand what possible pairs (or tuples) of 
these invariants can be realized for specific families 
of ideals, most notably,
ideals that are defined combinatorially (e.g.,
edge ideals, binomial edge ideals, and toric
ideals of graphs). This circle of problems
was introduced by Hibi, Higashitani, Kimura,
and O'Keefe \cite{HHKO}, who
compared the depth and dimension of 
toric ideals of graphs, and 
Hibi and Matsuda \cite{HM1,HM2}, who first showed that for any pair $r,d$ of positive integers, there exists a (lexsegment) monomial ideal $I$
with $(r,d)=(\reg(S/I),\deg (h_{S/I}(t)))$. 
Hibi, Matsuda, and Van Tuyl \cite{HMVT} later showed
a similar result for the edge ideals of graphs.  An investigation of the possible pairs $(\reg(S/I),\deg(h_{S/I}(t)))$ for other families of ideals 
soon followed, most notably, for binomial edge
ideals \cite{HM3,KK} and toric ideals of graphs \cite{FKVT}. 
Other recent work has focused on determining the
possible pairs $(\reg(S/I),\pdim(S/I))$ \cite{HaH} and $(\depth(S/I),\dim(S/I))$ \cite{HKKMV,HKO} when $I$ is an edge ideal.

Hibi, Kimura, Matsuda, and Van Tuyl 
\cite{HKMVT} introduced a variation 
of this problem by asking if
we can completely describe all pairs 
$({\rm reg}(S/I),{\rm deg}(h_{S/I}(t)))$ if we restrict
to edge ideals of connected graphs on $n$ vertices; in the  case of  Cameron--Walker graphs, they were able to answer this
question. In a related direction, Erey and Hibi \cite{EH} completely described all the possible pairs $({\rm reg}(S/I),{\rm pdim}(S/I))$  as $I$ varies over all the edge ideals of connected bipartite 
graphs on $n$ vertices.  Recently, Ficarra and Sgroi \cite{FS} gave an almost complete description
of 
$({\rm reg}(S/I),{\rm pdim}(S/I))$ 
for the case of binomial
edge ideals of graphs on $n$ non-isolated vertices.
Inspired by  \cite{EH,FS},
in this paper we characterize the values of the invariants that can occur if $I$ is the
toric ideal of a
connected bipartite graph 
on $n$ vertices.  

We introduce some 
notation to describe our results.  Let $G = (V,E)$ be a 
finite simple graph with edge set
$E = \{e_1,\ldots,e_q\}$ and vertex set 
$V =\{v_1,\ldots,v_n\}$. Define a 
$\mathbb{K}$-algebra homomorphism
$\varphi:R = \mathbb{K}[e_1,\ldots,e_q]
\rightarrow \mathbb{K}[v_1,\ldots,v_n]$
by $\varphi(e_i) = v_jv_k$ if $e_i = \{v_j,v_k\} \in E$.  The {\it toric ideal
of $G$}, denoted $I_G$, is the kernel
of $\varphi$.
The graph $G$ is a {\it bipartite
graph} if there exists $V_1,V_2 \subseteq V$ such that
$V = V_1 \cup V_2$ and $V_1 \cap V_2 =\emptyset$ with the property that every $e\in E$ has an endpoint
in both $V_1$ and $V_2$. 
We  define ${\rm CBPT}_{{\rm reg},\deg, {\rm pdim},{\rm depth},\dim}(n)$
to be the set

\noindent
\footnotesize
\[\left\{({\rm reg}(R/I_G),\deg (h_{R/I_G}(t)),
{\rm pdim}(R/I_G), {\rm depth}(R/I_G), \dim(R/I_G)) ~\left|~
\begin{array}{c}
\mbox{$G$ is a connected bipartite} \\
\mbox{graph on $n$ vertices}
\end{array}
\right\}\right..\]
\normalsize

Our main result is a complete characterization
of the elements in the above set:
\begin{theorem}\label{maintheorem}
Let $n \geq 2$.  Then ${\rm CBPT}_{{\rm reg},\deg, {\rm pdim},{\rm depth},\dim}(n)$ is given by
$$\left\{(r,r,p,n-1,n-1) ~\left|~ 
0 < r < \left\lfloor \frac{n}{2} \right\rfloor,~
1 \leq p \leq r(n-2-r) \right\}\right. \cup \{(0,0,0,n-1,n-1)\}.$$
\end{theorem}

\noindent
Notice that Theorem \ref{maintheorem} 
describes all five invariants;  the
only other result similar to Theorem \ref{maintheorem} is a result
of Hibi, Kanno, Kimura, Matsuda, and Van Tuyl \cite{HKKMV}
which described four of these invariants
for the edge ideals of Cameron--Walker
graphs.

The proof of Theorem \ref{maintheorem} 
requires both new and old results.
Recent work of 
Almousa, Dochtermann, and Smith \cite{ADS}
allow us to  bound the
regularity of toric
ideals of bipartite graphs
by looking at subgraphs.  We also require
an old graph theory result of Jackson \cite{J}
on the existence of cycles in bipartite graphs, which
enables us to find an upper bound 
on the number of edges of $G$ in terms of the regularity
of $R/I_G$ (see
Lemma \ref{edge lemma}).
Note that it is well-known that
the toric ideals of bipartite graphs are
Cohen--Macaualy; consequently,
 $\dim(R/I_G) =
{\rm depth}(R/I_G) = n-1$ (cf. {\cite[Corollary 10.1.21]{Villarreal-book}}) and 
$\deg(h_{R/I_G}(t))
= {\rm reg}(R/I_G)$ (cf. \cite[Corollary B.28]{Vas}).
Consequently, proving Theorem \ref{maintheorem}
is equivalent to proving 
what values the regularity and projective
dimension can obtain;
these details are given in Theorem \ref{reg vs pdim}.

Our paper is structured as follows.
In Section 2 we recall the relevant 
background on graph theory, commutative
algebra, and toric ideals of graphs.  In
Section 3, we focus on the regularity 
and projective dimension of
toric ideals of bipartite graphs.  
In Section 4 we
combine the previous results of the paper
to prove Theorem \ref{maintheorem}.

\noindent
{\bf Acknowledgements.} 
Some of our results first appeared in the 
first author’s MSc thesis \cite{B}.
The authors thank Anton Dochtermann, 
Antonino Ficarra, Huy T\`ai H\`a, 
Takayuki Hibi, Akihiro Higashitani, Irena Peeva, Ben Smith, and Rafael Villarreal 
for answering our questions and feedback.
Calculations using {\em Macaulay2} \cite{M2} 
inspired some of these results. Van Tuyl 
acknowledges the support of 
NSERC RGPIN-2019-05412.

\section{Background}
Throughout this paper $\mathbb{K}$ will denote an algebraically
closed field of characteristic zero.  In this section we collect
the facts needed to prove our main results.

\subsection{Graph theory} We recall
the relevant graph theory terminology.
A {\it finite simple graph} (or a {\it graph}) $G=(V(G),E(G))$ consists 
of a non-empty finite set 
$V(G) = \{v_1,\ldots,v_n\}$, called the
vertices, and a finite set 
$E(G) = \{e_1,\ldots,e_q\} \subseteq \{\{u,v\}\mid 
u,v \in V(G), u\neq v\}$ of distinct unordered pairs 
of distinct elements of $V(G)$, called 
the edges.  We sometimes write $V$ (resp. $E$) for $V(G)$ (resp. $E(G)$) if $G$ is clear from the context.
A graph $H = (V(H),E(H))$ is said to be a {\it subgraph} of a graph $G$ if $V(H) \subseteq V(G)$ and  $E(H) \subseteq E(G)$. In this case, we say that $G$ contains $H$ and write $H \subseteq G$.

A {\it walk} of $G$ is a sequence of edges $w = (e_1, e_2, \dots, e_m)$, where each $e_i=\{u_{i_1},u_{i_2}\}\in E$ and $u_{i_2}=u_{{(i+1)}_1}$ for each $i = 1, \dots, m - 1$. 
Equivalently, a walk is a sequence of vertices $(u_1, \dots,u_m,u_{m+1})$ such that 
$\{u_i,u_{i+1}\} \in E$ for all 
$i = 1,\dots,m$. Here, $m$ is referred to as the 
{\it length} of the walk. A walk is 
{\it even} if $m$ is even, and it is {\it closed} if $u_{m+1}=u_1$. 
    Two vertices $u$ and $v$ are said to be {\it connected} if there is a walk between them. A graph $G$ is said to be {\it connected} if every two distinct vertices of $G$ are connected.  A {\it connected
    component} of $G$ is a maximal
    connected subgraph of $G$.

A {\it cycle} of a graph $G$ is a closed walk $(u_1, \dots,u_m,u_{m+1}=u_1)$ of vertices of $G$ (with $m \geq 3$) such that the only vertices in the walk that are not pairwise distinct are $u_1$ and $u_{m+1}$. A cycle of length $m$ is called an 
{\it $m$-cycle}.   The {\it $m$-cycle graph},
denoted $C_m$, is the graph with the
vertex set $V(C_m) = \{v_1,\ldots,v_m\}$ and 
edge set $E(C_m) = \{\{v_1,v_2\},\{v_2,v_3\},
\ldots,\{v_{m-1},v_m\},\{v_m,v_1\}\}$.
A {\it tree} is a connected graph that contains no cycles; a 
{\it forest}
is graph where each connected component is
a tree.

We are primarily interested in bipartite graphs.
A graph $G=(V,E)$ is  {\it bipartite} if there exists a partition (or bipartition)
$V=V_1 \cup V_2$ with $V_1,V_2 \subseteq V$
and $V_1 \cap V_2 = \emptyset$ such that every edge of $E$ joins a vertex in $V_1$ and a vertex in $V_2$. Given $a,b \geq 1$, 
the {\it complete bipartite graph},
$K_{a,b}$ is the graph
with partition 
$V = \{x_1,\ldots,x_a\} \cup
\{y_1,\ldots,y_b\}$ and edge set
$\{\{x_i,y_j\} ~|~ 1 \leq i \leq a,\, 1 \leq j \leq b\}$.  Note that $C_m$ is a bipartite graph
if and only if $m$ is even.   All trees and forests
are also bipartite graphs.

A {\it matching} of a graph $G$ is a collection of pairwise non-adjacent edges of $G$. The {\it matching
number} of $G$, denoted $\mat(G)$, is the largest size of any matching of $G$.  Note that if $G$ 
is a bipartite graph, then it is straightforward to
verify that
\begin{equation}\label{matchbound}
\mat(G) \leq \left\lfloor \frac{|V(G)|}{2} \right\rfloor.
\end{equation}

The following result, which
determines the existence of cycles in a bipartite graph, will
play a pivotal role in later results.

\begin{theorem}[{\cite[Theorem 3]{J}}] \label{Jackson lemma}
    Let $m \geq 2$ be an integer, and 
    let $G$ be a bipartite graph with bipartition $V = A \cup B$, where $|A|=a$, $|B|=b$, and $2 \leq m \leq a \leq b$. If
    \[|E(G)|>\begin{cases}
  a+(b-1)(m-1), & a \leq 2m-2, \\
  (a+b -2m+3)(m-1), & a \geq 2m-2,
\end{cases}
\]
then $G$ contains a cycle of length at least $2m$.
\end{theorem}

\subsection{Commutative algebra}
Our goal is to compare a number of homological
invariants.  We recall their definitions and
some properties.
Let $S = \mathbb{K}[x_1,\dots,x_m]$ and let $I$
be a homogeneous ideal of $S$.
The {\it minimal graded free resolution
of $S/I$} has the form:
\begin{equation*}
\resizebox{1\hsize}{!}{$
0 \rightarrow \bigoplus_{j} S(-j)^{\beta_{p,j}(S/I)} 
\rightarrow   \bigoplus_{j} S(-j)^{\beta_{p-1,j}(S/I)} \rightarrow \cdots 
\rightarrow   \bigoplus_{j} S(-j)^{\beta_{1,j}(S/I)} \rightarrow S \rightarrow S/I \rightarrow 0$}
\end{equation*}
where $\beta_{i,j}(S/I)$ denotes the 
{\it $(i,j)$-th graded Betti number} of $S/I$.

Two of the invariants that we will consider are encoded
into this resolution.
The {\it \textup(Castelnuovo--Mumford\textup) regularity} of $S/I$ is defined to be $$\reg(S/I)= \max \{j-i \mid \beta_{i,j}(S/I) \neq 0 \}.$$
The {\it projective dimension} of $S/I$ is the length of the minimal graded free resolution, i.e., \[\pdim(S/I) = \max \{i \mid \beta_{i,j}(S/I) \neq 0  \text{~for some $j$} \}. \]

The {\it Hilbert series} of $S/I$ is the formal power series
\[
\textup{HS}_{S/I}(t) = \sum_{i\geq0} [\dim_\mathbb{K}(S/I)_i]t^i
\]
where $\dim_\mathbb{K}(S/I)_i$ is the dimension of the $i$-th graded piece of $S/I$.

The {\it depth} of $S/I$,
denoted ${\rm depth}(S/I)$, is the length of
any maximal regular sequence of $S/I$ that is contained in the maximal ideal $\mathfrak{m}=\langle x_1,\dots,x_m\rangle \subset S$.
The  {\it \textup(Krull\textup) dimension}  of 
 $S/I$,
denoted $\dim(S/I)$, is the supremum of the lengths of all chains of prime
ideals in $S/I$.  It is well  known that ${\rm depth}(S/I) \leq \dim(S/I)$ for
all homogeneous ideals $I$ of $S$.  We say
$S/I$ is {\it Cohen--Macaulay} if 
${\rm depth}(S/I) = \dim(S/I)$.

By the Hilbert--Serre
Theorem \cite[Corollary 4.1.8]{BH}, there exists a unique polynomial $h_{S/I}(t) \in \Z[t]$, called the {\it $h$-polynomial} of $S/I$, such that $\textup{HS}_{S/I}(t)$ can be written as \[
    \textup{HS}_{S/I}(t)= \frac{h_{S/I}(t)}{(1-t)^{\dim(S/I)}}
    \] with $h_{S/I}(1) \neq 0$. We denote the {\it degree of the $h$-polynomial} $h_{S/I}(t)$  by $\deg (h_{S/I}(t))$.

We want to
compare the invariants
${\rm reg}(S/I)$, ${\rm pdim}(S/I)$,
${\rm depth}(S/I)$, $\dim(S/I)$,
and $\deg(h_{S/I}(t))$.  The following result provides some useful
relations among these invariants.

\begin{theorem}\label{Auslander}
Let $I$ be a proper homogeneous
ideal of $S = \mathbb{K}[x_1,\ldots,x_m]$. Then
\begin{enumerate}
    \item[$(i)$]  $\pdim(S/I) + \depth(S/I) = m$;
    \item[$(ii)$] if $S/I$ is Cohen-Macaulay, then  ${\rm reg}(S/I) =
    \deg(h_{S/I}(t))$.
    \end{enumerate}
\end{theorem}

\begin{proof}
Statement $(i)$ is a special
case of the Auslander--Buchsbaum
formula {\cite[Theorem 1.3.3]{BH}}.
Statement $(ii)$ is \cite[Corollary B.28]{Vas}
(or Corollary B.4.1 in earlier
printings).
\end{proof}

\subsection{Toric ideals of graphs}
We define the family of ideals that are studied
in this paper and some of their
properties. Let $G$ be a graph with vertex set $V=\{v_1,\dots,v_n\}$ and edge set $E=\{e_1,\dots,e_q\}$ with $q \geq 1$. 
Let $\K[V]=\K[v_1,\dots,v_n]$ and $\K[E]=\K[e_1,\dots,e_q]$ be polynomial rings in the vertex and edge variables, respectively. Define a
$\K$-algebra homomorphism $\varphi\colon \K[E] \to \K[V]$ by 
$\varphi(e_i) = v_{i_1} v_{i_2}$ for all $e_i = \{v_{i_1}, v_{i_2}\} \in E$, $1 \leq i \leq q$.
The {\it toric ideal of $G$}, denoted $I_G$, is defined to be the kernel of the homomorphism $\varphi$.

\begin{remark} 
    The ideal $I_G$ is called a toric ideal
    since $I_G$ is a prime binomial ideal.  Indeed,
    the image of $\varphi$ is an 
    integral domain, and since $\varphi(\K[E])$ is 
    isomorphic to $\K[E]/I_G$ by the first 
    isomorphism theorem, it follows that $I_G$ is a 
    prime ideal.   Theorem 
    \ref{toricgenerators}, which is given below, 
    shows that $I_G$ is  a binomial ideal.  Note that in  the definition of $I_G$, we avoid 
    the case that $G$ has no edges to ensure
    that $\K[E]$ has at least one variable.
    \end{remark}

    \begin{remark}
            We write $\K[G]$ to denote the quotient ring $\K[E]/I_G$. Note that in the literature (see e.g., \cite{ADS, HH,HHKO}), $\K[G]$ often denotes the \textit{edge ring of $G$} (i.e., the image $\im(\varphi)$ of $\varphi$). As mentioned in the previous remark, $\im(\varphi)$ and $\K[E]/I_G$ are isomorphic as rings; however, we must take care when stating results about gradings on these rings, as they may differ. In all subsequent appearances of the notation $\K[G]$, we have ensured that results from the literature concerning $\K[G]$ remain true under our interpretation.
    \end{remark}   

While the generators of $I_G$ are defined
implicitly, there is a well-known 
connection between the closed even walks of a graph $G$ and a (possibly non-minimal) set of  generators for $I_G$.  For 
a closed even walk $\Gamma=(e_{i_1},\dots,e_{i_{2m}})$ of graph $G$, we define a binomial 
\[f_\Gamma = e_{i_1}e_{i_3}\cdots e_{i_{2m-1}}
- e_{i_2}e_{i_4}\cdots e_{i_{2m}}.\]
We can now describe a set of generators of $I_G$.
    \begin{theorem}[{\cite[Proposition 3.1]{Villarreal-article}}]
\label{toricgenerators}
If $I_G$ is the toric ideal of a graph $G$, then
\[I_G = \langle f_\Gamma ~|~ \mbox{$\Gamma$ is
a closed even walk of $G$}\rangle.\]
If $G$ is bipartite, then $I_G =
 \langle f_\Gamma ~|~ \mbox{$\Gamma$ is
a even cycle of $G$}\rangle.$
\end{theorem}

We now collect together some facts about
the toric ideals of bipartite graphs.  We need
a result due to Almousa, Dochtermann, and Smith \cite{ADS} that  allows us to find  bounds on the regularity using subgraphs (while \cite{ADS} includes the connected hypothesis, it
can be shown that this hypothesis is not required; for
our purposes, we only require the original statement).

\begin{theorem}[{\cite[Theorem 6.11]{ADS}}] \label{reg of subgraph} Suppose $G \subseteq K_{a,b}$ is a connected bipartite graph and let $G' \subseteq G$ be a connected subgraph with at least two vertices. Then $\reg(\K[G']) \leq \reg(\K[G])$.
\end{theorem}

The next result summarizes some useful results in the literature.

\begin{theorem}\label{matching}
Let $G$ be a connected bipartite graph on
$n \geq 2$ vertices. 
\begin{enumerate}
\item[$(i)$] 
\cite[Theorem 1]{HH}
or \cite[Proposition 14.4.19]{Villarreal-book}
    $\reg(\K[G]) \leq \mat(G) - 1$.
    \item[$(ii)$] {\cite[Corollary 10.1.21]{Villarreal-book}}
    $\dim(\K[G])=n-1$.
    \item[$(iii)$] {\cite[Corollary 5.26]{HHO}}  $\K[G]$ is Cohen--Macaulay.
    \end{enumerate}
\end{theorem}

We get the following useful facts
as corollaries.

\begin{corollary}\label{reg upper/lower bound}\label{p=q-n+1}
Let $G$ be a connected bipartite graph
on $n \geq 2$ vertices with $q$ edges. 
\begin{enumerate}
\item[$(i)$] $\depth(\K[G]) = \dim(\K[G]) = n-1$.
\item[$(ii)$] $0\leq \deg(h_{\K[G]}(t)) = 
\reg(\K[G]) < \left\lfloor \frac{n}{2} \right\rfloor$.
\item[$(iii)$] $0\leq {\pdim}(\K[G]) = q-n+1$.
\end{enumerate}
\end{corollary}
\begin{proof}
By Theorem \ref{matching} $(iii)$, the ring
$\K[G]$ is Cohen--Macaulay.  Statement
$(i)$ now follows from the definition of
Cohen--Macaulayness and Theorem
\ref{matching} $(ii)$.   For statement
$(ii)$, 
because $\K[G]$
is Cohen--Macaulay, Theorem \ref{Auslander} $(ii)$
implies that  $\deg(h_{\K[G]}(t)) = 
\reg(\K[G])$.  The inequality 
follows from Theorem \ref{matching} $(i)$
and inequality \eqref{matchbound}.
For statement $(iii)$, Theorem \ref{Auslander} $(i)$
gives $\pdim(\K[G]) = q - \depth(\K[G]) = q-n+1$
since $\K[E]$ has $q$ variables.
\end{proof}

For some special families of graphs, we can 
give exact values for the regularity.

\begin{lemma} \label{reg of complete bipartite}\label{reg of cycle}
The following formulas hold:
\begin{enumerate}
    \item[$(i)$] if $G=K_{a,b}$, then $\reg(\K[G])=\min\{a,b\}-1$;
    \item[$(ii)$] if $G =C_{2r}$ with $r\geq 2$, then $\reg(\K[G]) = r-1$.
\end{enumerate}
\end{lemma}

\begin{proof}
    For $(i)$, see \cite[Corollary 4.11]{BOVT}. For $(ii)$,
    Theorem \ref{toricgenerators} implies $I_{C_{2r}}$
    is a principal ideal with a minimal generator of degree $r$. 
    The conclusion follows from this fact.
\end{proof}

We also require a classification
of the toric ideals of  bipartite graphs with regularity and projective dimension
equal to zero.

\begin{lemma}\label{reg of tree}
Let $G$ be a bipartite graph  on $n\geq2$ vertices.  Then the following are equivalent:
\begin{enumerate}
    \item[$(i)$] $G$ is a forest.
    \item[$(ii)$] ${\rm reg}(\K[G])=0$.
    \item[$(iii)$] ${\rm pdim}(\K[G]) =0$.
\end{enumerate}
\end{lemma}

\begin{proof}
If $G$ is a forest, then $G$ has no even cycles, so $I_G = \langle 0 \rangle$ by Theorem \ref{toricgenerators}.  Thus
${\rm reg}(\K[G]) = {\rm pdim}(\K[G]) = 0$.
Conversely, if $G$ is not a forest, then $G$ has at least 
one even cycle.  So, by Theorem \ref{toricgenerators}, 
$I_G \neq \langle 0 \rangle$, from which it
follows that both the regularity and projective
dimension of $\K[G]$ are non-zero.
\end{proof}

\begin{remark}
    The previous lemma is only classifying 
    bipartite graphs with regularity and projective dimension zero.  There are non-bipartite graphs whose toric ideals have
    regularity and projective dimension zero, e.g.,
    the toric ideals of  $C_{2r+1}$ with 
    $r \geq 1$.
\end{remark}

\begin{remark}
    While we have only highlighted the
    results about the regularity of toric
    ideals of bipartite graphs that  we require, we point 
    out that a number of other results are known
    from the perspective of the $a$-invariant (we thank
    Rafael Villarreal for pointing out this connection). Given a homogeneous 
    ideal 
    $I \subseteq S = \K[x_1,\ldots,x_m]$, the
    {\it $a$-invariant} of $S/I$, denoted $a(S/I)$,
    is the degree of ${\rm HS}_{S/I}(t)$ as a rational
    function (in lowest terms).  So, 
    $a(S/I) = \deg h_{S/I}(t) - \dim(S/I)$.  
    When $G$ is a connected bipartite graph,
    we can use Corollary 
    \ref{p=q-n+1} to show that
    $a(\K[G])= \reg(\K[G])-n+1$.  Consequently,
    studying the regularity of toric ideals of
    connected bipartite graphs is equivalent to studying their $a$-invariant.
    In \cite{VV},
    Valencia and Villarreal gave a
    combinatorial interpretation for the 
    $a$-invariant, and a linear program
    to compute this invariant.  
    Also see  \cite[Section 11.5]{Villarreal-book} for other properties of $a(\K[G])$;
    for example, Lemma \ref{reg of complete bipartite} $(i)$ is equivalent to    \cite[Corollary 11.5.2]{Villarreal-book} which computes the $a$-invariant
    for the toric ideal of $K_{a,b}$.
    Remark \ref{codegree} gives another
    combinatorial interpretation of ${\reg}(\K[G])$ for bipartite graphs.
\end{remark}

\section{Comparing the regularity and projective dimension}

Let $\text{CBPT}(n)$ denote the set of  connected bipartite
graphs on $n$ vertices.  In this section, we focus
on comparing the regularity and
projective dimension of the toric ideals of graphs
$G$ with $G \in \text{CBPT}(n)$.  In particular, we
describe the set
\[\text{CBPT}^\text{pdim}_\text{reg} (n) = \{(\reg(\K[G]), \pdim(\K[G])) : G \in \text{CBPT}(n)\}.\] 
Understanding this set will be key to proving
Theorem \ref{maintheorem}.

We begin with a simple inequality that we will use in subsequent lemmas.
\begin{lemma}\label{r+1 leq n-r-1}
    Let $n$ and $r$ be integers. If $r\leq \left \lfloor{\frac{n}{2}}\right \rfloor-1$, then $0 \leq n-2-2r$.
\end{lemma}
\begin{proof} 
Observe that $r \leq \left \lfloor{\frac{n}{2}}\right \rfloor-1 \leq \frac{n}{2}-1=\frac{n-2}{2}.$ Hence $2r \leq n-2$, so $0 \leq n-2-2r$.
\end{proof}
The following two lemmas give ranges of positive integers $r$ and $p$ that can be realized as $r=\reg(\K[G])$ and $p=\pdim(\K[G])$ for some connected bipartite graph $G$. More precisely, Lemma \ref{lower existence} (resp. Lemma \ref{upper existence}) shows that any $n\geq4$, $0 < r < \left \lfloor{\frac{n}{2}}\right \rfloor$ and  $1\leq p \leq r^2$ (resp. $r^2 \leq p \leq r(n-2-r)$) can be realized as $r=\reg(\K[G])$ and $p=\pdim(\K[G])$ for some connected bipartite graph $G$ on $n$ vertices. (Note that the $p=r^2$ case is covered twice.)
\begin{lemma}\label{lower existence}
Let $n$, $r$, $p$ be integers with $n\geq 4$, $0< r < \left \lfloor{\frac{n}{2}}\right \rfloor$, and $1\leq p \leq r^2$. Then there exists a connected bipartite graph $G$ on $n$ vertices with $\reg(\K[G])=r$ and $\pdim(\K[G])=p$. 
\end{lemma}
\begin{proof}
We describe the construction of the desired graph. Letting $n$, $r$, $p$ be as given, we define the bipartite
graph $G_{n,r,p}$ as follows. Let our vertex set $V$ be 
$$V = \{x_1,\ldots,x_{r+1},y_1,\ldots,y_{r+1},z_1,\ldots,z_{n-2r-2}\}$$
with bipartition $V_1 = \{x_1,\ldots,x_{r+1}\}$ and
$V_2 = \{y_1,\ldots,y_{r+1},z_1,\ldots,z_{n-2r-2}\}$.
To define our edge set, let $E_1$ and $E_2$ be
\begin{align*}
E_1 & = \{\{x_1,y_1\},\{y_1,x_2\},\{x_2,y_2\},\{y_2,x_3\},
\ldots, \{x_{r+1},y_{r+1}\},\{y_{r+1},x_1\}\}; \\
E_2 &= \{\{x_{r+1},z_j\} ~|~ 1 \leq j \leq n-2r-2\}.
\end{align*}
(Note that $n-2r-2 \geq0$ by Lemma \ref{r+1 leq n-r-1}. If $n-2r-2 = 0$, then $E_2 = \emptyset$ and
there are no $z_i$ vertices.)
Note $E_1$ is a cycle of length $2r+2$.  Let 
$E_3$ be any $p-1 \geq 0$ edges with one vertex 
in $\{x_1,\ldots,x_{r+1}\}$ and the other
in $\{y_1,\ldots,y_{r+1}\}$ that do not already 
appear in $E_1$.  Because we can have at most
$(r+1)^2 = r^2+2r+1$ edges between the $x_i$'s and
$y_j$'s, and since 
$E_1$ already used $2r+2$ of these edges,  there
are only $r^2-1 = r^2+2r+1-(2r+2)$ possible
choices for these $p-1$ edges.  Since $p \leq r^2$,
it is possible to find $p-1$ such edges. Let $E = E_1 \cup
E_2 \cup E_3$ be the edge set of $G_{n,r,p}$.
 See Figure \ref{fig:lower} for an example of the graph $G_{n,r,p}$ for $n=10, r=3$, and $p=2$. We claim that $G=G_{n,r,p}$ is a connected bipartite graph on $n$ vertices
 with $\reg(\K[G])=r$ and $\pdim(\K[G])=p$. 

By construction $G$ is 
a connected bipartite graph on $n = (r+1)+(r+1)+(n-2r-2)$ vertices with $q = 2r+2+(n-2r-2)+
(p-1) = n+p-1$ edges.
Observe that $G$ is a subgraph of the 
complete bipartite graph $K_{r+1,n-r-1}$.
Since $r+1 \leq n-r-1$ by Lemma \ref{r+1 leq n-r-1}, by Theorem \ref{reg of subgraph} and
Lemma \ref{reg of complete bipartite} $(i)$,
we have
\[
\reg(\K[G]) \leq \reg(\K[K_{r+1,n-r-1}]) = \min \{r+1,n-r-1\} -1= r.
\]
Since $G$ contains the cycle $C_{2r+2}$ as a subgraph (the edges of $E_1$ form this cycle) we have $r =  \reg(\K[C_{2r+2}]) \leq \reg(\K[G])$ by Theorem \ref{reg of subgraph} and Lemma \ref{reg of cycle} $(ii)$. Thus $\reg(\K[G])=r$.  Furthermore,
it follows from Corollary \ref{p=q-n+1} $(iii)$ that
\[\pdim(\K[G]) = q-n+1 = (n+p-1) -n+1 = p,\]
as desired.
\end{proof}

\begin{figure}[t]
\begin{tikzpicture}[scale=0.6]
 
\draw (0,0) -- (0,3) -- (3,0) -- (3,3) --
(6,0) -- (6,3) -- (9,0) -- (9,3) -- (0,0);
\draw (9,0) -- (12,3);
\draw (9,0) -- (15,3);
\draw (0,0) -- (3,3);

\fill[fill=white,draw=black] (0,0) circle (.2)
node[label=below:$x_1$] {};
\fill[fill=white,draw=black] (3,0) circle (.2)
node[label=below:$x_2$] {};
\fill[fill=white,draw=black] (6,0) circle (.2)
node[label=below:$x_3$] {};
\fill[fill=white,draw=black] (9,0) circle (.2)
node[label=below:$x_4$] {};
\fill[fill=white,draw=black] (0,3) circle (.2)
node[label=above:$y_1$] {};
\fill[fill=white,draw=black] (3,3) circle (.2)
node[label=above:$y_2$] {};
\fill[fill=white,draw=black] (6,3) circle (.2)
node[label=above:$y_3$] {};
\fill[fill=white,draw=black] (9,3) circle (.2)
node[label=above:$y_4$] {};
\fill[fill=white,draw=black] (12,3) circle (.2)
node[label=above:$z_1$] {};
\fill[fill=white,draw=black] (15,3) circle (.2)
node[label=above:$z_2$] {};

\end{tikzpicture}
\caption{The graph $G =G_{10,3,2}$
with $\reg(\K[G]) =3$ and $\pdim(\K[G])=2$
}\label{fig:lower}
\end{figure}
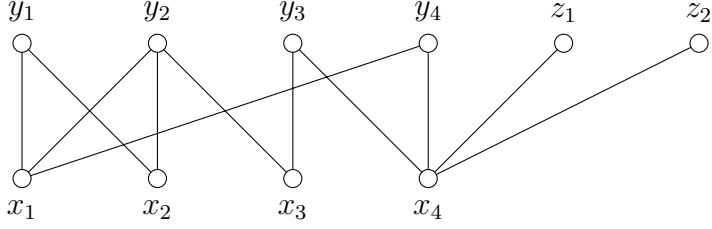
\begin{lemma}\label{upper existence}
Let $n$, $r$, $p$ be integers with $n\geq 4$, $0< r < \left \lfloor{\frac{n}{2}}\right \rfloor$, and $r^2 \leq p \leq r(n-2-r)$. Then there exists a connected  bipartite graph $G$ on $n$ vertices with $\reg(\K[G])=r$ and $\pdim(\K[G])=p$. 
\end{lemma}
\begin{proof}
We describe the construction of the desired graph.  Letting $n$, $r$, $p$ be as given, we define the bipartite
graph $H_{n,r,p}$ as follows. Let our vertex set $V$ be 
\[V = \{x_1,\ldots,x_{r+1},y_1,\ldots,y_{r+1},z_1,\ldots,z_{n-2r-2}\}\]
with bipartition $V_1 = \{x_1,\ldots,x_{r+1}\}$ and
$V_2 = \{y_1,\ldots,y_{r+1},z_1,\ldots,z_{n-2r-2}\}$.
To define our edge set, let $E_1$ and $E_2$ be
\begin{align*}
E_1 & = \{\{x_i,y_j\} ~|~ 1 \leq i,j \leq r+1\}; \\
E_2 &= \{\{x_{r+1},z_j\} ~|~ 1 \leq j \leq n-2r-2\}.
\end{align*}
(Note that $n-2r-2 \geq0$ by Lemma \ref{r+1 leq n-r-1}. If $n-2r-2 = 0$, then $E_2 = \emptyset$ and
there are no $z_i$ vertices.) Note that there can be at 
most $r(n-2r-2) = r(n-2-r) -r^2$ edges between $\{x_1,\ldots,x_{r}\}$
and $\{z_1,\ldots,z_{n-2r-2}\}$.
Let $E_3$ be a set 
containing any $p-r^2 \geq 0$ of these
edges. Observe
that our hypotheses imply that 
$r(n-2r-2) \geq p-r^2$ so it possible to find
$p-r^2$ such edges.
Let $E = E_1 \cup E_2
\cup E_3$ be the edge set of $H_{n,r,p}$.
See Figure \ref{fig:upper} for an example of the graph $H_{n,r,p}$ for $n=10, r=3$, and $p=12$. . We claim that $G=H_{n,r,p}$ is a connected bipartite graph on $n$ vertices with $\reg(\K[G])=r$ and $\pdim(\K[G])=p$.

By construction $G$ is  a connected bipartite graph on $n = (r+1)+(r+1)+(n-2r-2)$ vertices with 
$q = (r+1)^2+(n-2r-2)+(p-r^2)
= n+p-1$ edges. 
Observe that $G$ is a subgraph of the 
complete bipartite graph $K_{r+1,n-r-1}$.
Since $r+1 \leq n-r-1$ by Lemma \ref{r+1 leq n-r-1}, by Theorem \ref{reg of subgraph} and
Lemma \ref{reg of complete bipartite} $(i)$,
we have
\[
\reg(\K[G]) \leq \reg(\K[K_{r+1,n-r-1}]) = \min \{r+1,n-r-1\} -1= r.
\]
On the other hand, note that $G$ contains the complete bipartite graph $K_{r+1,r+1}$ 
on $\{x_1,\ldots,x_{r+1},y_1,\ldots,y_{r+1}\}$.
Theorem \ref{reg of subgraph} and Lemma
\ref{reg of complete bipartite} $(i)$ imply that 
$r \leq \reg(\K[G])$.  Thus $\reg(\K[G])=r$, as 
desired.
 Furthermore,
it follows from Corollary \ref{p=q-n+1} $(iii)$ that
\[\pdim(\K[G]) = q-n+1=(n+p-1)-n+1 = p,\] completing the proof.
\end{proof}

\begin{figure}[t]
\begin{tikzpicture}[scale=0.6]
 
\draw (0,0) -- (0,3) -- (3,0) -- (3,3) --
(6,0) -- (6,3) -- (9,0) -- (9,3) -- (0,0);
\draw (9,0) -- (12,3);
\draw (9,0) -- (15,3);
\draw (0,0) -- (3,3);
\draw (0,0) -- (6,3);
\draw (0,0) -- (9,3);
\draw (3,0) -- (6,3);
\draw (3,0) -- (9,3);
\draw (6,0) -- (0,3);
\draw (6,0) -- (9,3);
\draw (9,0) -- (0,3);
\draw (9,0) -- (3,3);
\draw (0,0) -- (12,3);
\draw (0,0) -- (15,3);
\draw (3,0) -- (12,3);
\draw (3,0) -- (15,3);

\fill[fill=white,draw=black] (0,0) circle (.2)
node[label=below:$x_1$] {};
\fill[fill=white,draw=black] (3,0) circle (.2)
node[label=below:$x_2$] {};
\fill[fill=white,draw=black] (6,0) circle (.2)
node[label=below:$x_3$] {};
\fill[fill=white,draw=black] (9,0) circle (.2)
node[label=below:$x_4$] {};
\fill[fill=white,draw=black] (0,3) circle (.2)
node[label=above:$y_1$] {};
\fill[fill=white,draw=black] (3,3) circle (.2)
node[label=above:$y_2$] {};
\fill[fill=white,draw=black] (6,3) circle (.2)
node[label=above:$y_3$] {};
\fill[fill=white,draw=black] (9,3) circle (.2)
node[label=above:$y_4$] {};
\fill[fill=white,draw=black] (12,3) circle (.2)
node[label=above:$z_1$] {};
\fill[fill=white,draw=black] (15,3) circle (.2)
node[label=above:$z_2$] {};

\end{tikzpicture}
\caption{The graph $G= H_{10,3,12}$ with $\reg(\K[G]) = 3$ and  
$ \pdim(\K[G]) = 12$}\label{fig:upper}
\end{figure}
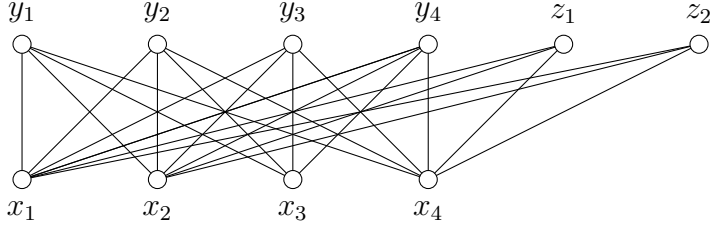

The next result, which is of independent
interest, provides an upper  bound on the number
of edges in a bipartite graph $G$ in terms
of the regularity.

\begin{lemma}\label{edge lemma}
    Let $G$ be a connected bipartite graph on $n \geq 2$ vertices. If $\reg(\K[G])=r$, then $G$ has at most $(r+1)(n-r-1)$ edges.
\end{lemma}
\begin{proof}
    Suppose towards a contradiction that $|E(G)|>(r+1)(n-r-1)$. Since $G$ is bipartite we can consider $G$ as a subgraph of $K_{a,b}$ 
    with $1\leq a \leq b$ and $a+b=n$. Observe that if $a=1$, then $G$ is a tree since $G$ is connected. Hence $r=0$ by Lemma \ref{reg of tree}, so $|E(G)|=n-1 =(r+1)(n-r-1)$ and we obtain a contradiction.
    
    So we may assume that $2 \leq a \leq b$. Since $G \subseteq K_{a,b}$ and $\reg(\K[K_{a,b}])=a-1$ by Lemma \ref{reg of complete bipartite} $(i)$, it follows by Theorem \ref{reg of subgraph} that $a\geq r+1$. If $a=r+1$, then $b=n-r-1$ and so $G$ is a subgraph of $K_{r+1,n-r-1}$ which has  $(r+1)(n-r-1)$ edges, contradicting the assumption that $|E(G)| >(r+1)(n-r-1)$. So we can assume $a \geq r+2$. Hence $2 \leq r+2 \leq a \leq b$, and so we can apply Theorem \ref{Jackson lemma}. We have two cases:

{\bf Case 1: $a \leq 2(r+2)-2$.} Since $0 \leq a-r-1$, we have $b \leq a+b-r-1=n-r-1$. Hence $br \leq rn-r^2-r$, so 
\[br + (n-r-1) \leq rn-r^2-r + (n-r-1) = (r+1)(n-r-1).\]
It follows that $a+(b-1)(r+1) =br+ (n-r-1) \leq (r+1)(n-r-1)<|E(G)|$, so we conclude by Theorem \ref{Jackson lemma} that $G$ contains a cycle of length at least $2(r+2)$.

{\bf Case 2: $a \geq 2(r+2)-2$.} Then $(a+b-2(r+2)+3)(r+1)=(n-2r-1)(r+1) \leq (n-r-1)(r+1) < |E(G)|$, so again we conclude by Theorem \ref{Jackson lemma} that $G$ contains a cycle of length at least $2(r+2)$.

In either case, $G$ contains an even cycle
$C$ of length at least $2(r+2)$.
But then $\reg(\K[G]) \geq \reg(\K[C]) \geq r+1$ by Theorem \ref{reg of subgraph} and Lemma \ref{reg of cycle} $(ii)$, contradicting the fact that $\reg(\K[G])=r$. This final contradiction concludes the proof.
\end{proof}

We now arrive at the main result of this section.

\begin{theorem} \label{reg vs pdim}
Let $n\geq2$ be an integer. Then 
\[\textnormal{CBPT}^\textnormal{pdim}_\textnormal{reg} (n) = \Big \{(r, p) \in {\mathbb{Z}}^2 \mid 0 < r < \left \lfloor{\frac{n}{2}}\right \rfloor,\; 1 \leq p \leq r(n-2-r) \Big \} \cup \{(0,0)\}.\]
\end{theorem}
\begin{proof}
We show both inclusions, starting with $\supseteq$.
If $n \in \{2,3\}$, then the RHS set is $\{(0,0)\}$, and by Lemma ~\ref{reg of tree}, we know that there is a tree $G$ on $n$ vertices with $\reg(\K[G])=\pdim(\K[G])=0$. So suppose $n \geq 4$ and take an element $(r,p)$ of the RHS set. If $(r,p)=(0,0)$, then again by Lemma ~\ref{reg of tree}, there is a tree $G$ on $n$ vertices with $\reg(\K[G])=\pdim(\K[G])=0$. Otherwise, if $(r,p)\neq(0,0)$, we must have  $0< r < \left \lfloor{\frac{n}{2}}\right \rfloor$, and $1 \leq p \leq r(n-2-r)$. Lemma ~\ref{lower existence} and Lemma
~\ref{upper existence}  imply that there is a connected bipartite graph $G$ on $n$ vertices with $\reg(\K[G])=r$ and $\pdim(\K[G])=p$, which verifies the first inclusion. 

Now, let $n\geq 2$ and $G\in \textnormal{CBPT}(n)$. By Corollary~\ref{reg upper/lower bound} $(ii)$, we have $0 \leq \reg(\K[G])< \left \lfloor{\frac{n}{2}}\right \rfloor$. If $\reg(\K[G])=0$, then $\pdim(\K[G])=0$ by Lemma \ref{reg of tree}. Hence letting $r=\reg(\K[G])$ and $p=\pdim(\K[G])$, we just need to show that if $0< r < \left \lfloor{\frac{n}{2}}\right \rfloor$, then $1 \leq \pdim(\K[G]) \leq r(n-2-r)$. So suppose $0 < r <\left \lfloor{\frac{n}{2}}\right \rfloor$. Since $r \neq 0$, we must have $1 \leq p$ by Lemma \ref{reg of tree}.

It remains to show that $p \leq r(n-2-r)$. By Lemma~\ref{edge lemma}, we have that $q \leq (r+1)(n-r-1)$, where $q$ is the number of edges of $G$. Hence Corollary \ref{p=q-n+1} $(iii)$ gives
\begin{align*}
p=q-n+1 
&\leq (r+1)(n-r-1)-n+1 \\
&=rn-r^2-2r \\
&=r(n-r-2)
\end{align*}
as desired, which concludes the proof.
\end{proof}
As an example of Theorem \ref{reg vs pdim},     Figure \ref{fig:reg vs pdim} shows the sets $\textnormal{CBPT}^\textnormal{pdim}_\textnormal{reg} (n)$ for $n=8,9$.
\begin{figure}
 \centering
\begin{tikzpicture}[thick,scale=0.85, every node/.style={scale=0.8}]
\draw[thin,->] (-0.5,0) -- (2.5,0) node[right] {$r = {\rm reg}$};
\draw[thin,->] (0,-0.5) -- (0,7.0) node[above] {$p = {\rm pdim}$};

\foreach \x [count=\xi starting from 0] in {1,2,3,4,5,6,7,8,9,10,11,12,13}{
        \draw (1pt,\x/2) -- (-1pt,\x/2);
    \ifodd\xi
        \node[anchor=east] at (0,\x/2) {$\x$};
    \fi
 \draw (1/2,1pt) -- (1/2,-1pt);
 \draw (2/2,1pt) -- (2/2,-1pt);
 \draw (3/2,1pt) -- (3/2,-1pt);
     \node[anchor=north] at (1/2,0) {$1$};
     \node[anchor=north] at (2/2,0) {$2$};
     \node[anchor=north] at (3/2,0) {$3$};
}

\foreach \point in {(0,0),(1/2,1/2),(1/2,2/2),(1/2,3/2),(1/2,4/2),(1/2,5/2),
(2/2,1/2),(2/2,2/2),(2/2,3/2),(2/2,4/2),(2/2,5/2),(2/2,6/2),(2/2,7/2),(2/2,8/2),
(3/2,1/2),(3/2,2/2),(3/2,3/2),(3/2,4/2),(3/2,5/2),(3/2,6/2),(3/2,7/2),(3/2,8/2),(3/2,9/2)}{
    \fill \point circle (2pt);
}

\draw[thin,->] (5.5,0) -- (8.5,0) node[right] {$r = {\rm reg}$};
\draw[thin,->] (6,-0.5) -- (6,7.0) node[above] {$p ={\rm pdim}$};

\foreach \x [count=\xi starting from 0] in {1,2,3,4,5,6,7,8,9,10,11,12,13}{
\draw ([xshift=1pt] 6,\x/2) -- ([xshift=-1pt] 6,\x/2);
    \ifodd\xi
        \node[anchor=east] at (5.9,\x/2) {$\x$};
    \fi
\draw (6+1/2,1pt) -- (6+1/2,-1pt);
 \draw (6+2/2,1pt) -- (6+2/2,-1pt);
 \draw (6+3/2,1pt) -- (6+3/2,-1pt);
     \node[anchor=north] at (6+1/2,0) {$1$};
     \node[anchor=north] at (6+2/2,0) {$2$};
     \node[anchor=north] at (6+3/2,0) {$3$};
       }

\
\foreach \point in {(6,0),(6+1/2,1/2),(6+1/2,2/2),(6+1/2,3/2),(6+1/2,4/2),(6+1/2,5/2),(6+1/2,6/2),
(6+2/2,1/2),(6+2/2,2/2),(6+2/2,3/2),(6+2/2,4/2),(6+2/2,5/2),(6+2/2,6/2),(6+2/2,7/2),(6+2/2,8/2),(6+2/2,9/2),(6+2/2,10/2),
(6+3/2,1/2),(6+3/2,2/2),(6+3/2,3/2),(6+3/2,4/2),(6+3/2,5/2),(6+3/2,6/2),(6+3/2,7/2),(6+3/2,8/2),(6+3/2,9/2),(6+3/2,10/2),(6+3/2,11/2),(6+3/2,12/2)}{
    \fill \point circle (2pt);
}
\node at (3,4) {$n=8$};
\node at (9,4) {$n=9$};
\end{tikzpicture}
\caption{Possible $(r,p) = (\reg(\K[G]),\pdim (\K[G]))$
  for all connected bipartite graphs on 8 and 9 vertices}\label{fig:reg vs pdim}
\end{figure}
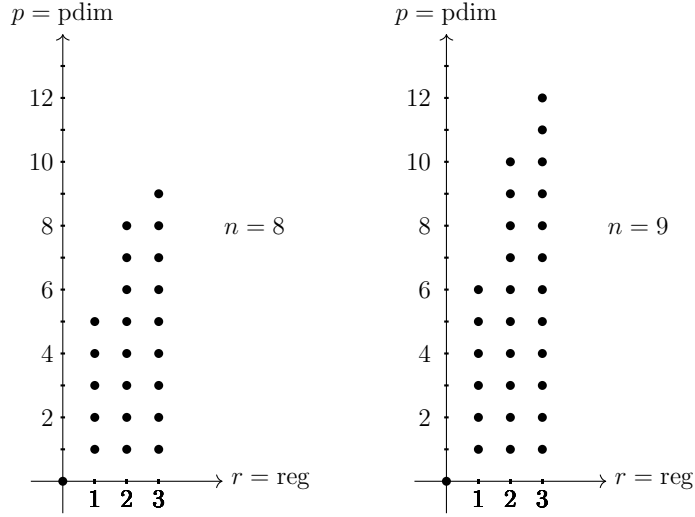
Now that we have a description of $\textnormal{CBPT}^\textnormal{pdim}_\textnormal{reg} (n)$ for each $n$, we can compute its cardinality.
\begin{corollary} \label{counting pdim}
For each $n\geq 2$, 
\begin{align*}
|\textnormal{CBPT}^\textnormal{pdim}_\textnormal{reg} (n)|&= 1-\frac{1}{6}\left \lfloor{\frac{n}{2}}\right \rfloor \Big(\left \lfloor{\frac{n}{2}}\right \rfloor-1 \Big) \Big(2 \left \lfloor{\frac{n}{2}}\right \rfloor-3n+5 \Big). 
\end{align*}
\end{corollary}
\begin{proof}
 This can be checked directly for $n=2,3$. 
    Let $n \geq 4$.  By Theorem \ref{reg vs pdim}, we have 
\begin{align*}
|\textnormal{CBPT}^\textnormal{pdim}_\textnormal{reg} (n)| - 1 &= \sum_{r=1}^{\left \lfloor{\frac{n}{2}}\right \rfloor -1} r(n-2-r) = (n-2)\sum_{r=1}^{\left \lfloor{\frac{n}{2}}\right \rfloor -1} r - \sum_{r=1}^{\left \lfloor{\frac{n}{2}}\right \rfloor -1} r^2 \\
&=\frac{n-2}{2}\left \lfloor{\frac{n}{2}}\right \rfloor \Big(\left \lfloor{\frac{n}{2}}\right \rfloor-1 \Big) -\frac{1}{6}\left \lfloor{\frac{n}{2}}\right \rfloor \Big(\left \lfloor{\frac{n}{2}}\right \rfloor-1 \Big)\Big (2 \left \lfloor{\frac{n}{2}}\right \rfloor-1 \Big) \\
&=-\frac{1}{6}\left \lfloor{\frac{n}{2}}\right \rfloor \Big(\left \lfloor{\frac{n}{2}}\right \rfloor-1 \Big) \Big(2 \left \lfloor{\frac{n}{2}}\right \rfloor-3n+5 \Big).
\qedhere
\end{align*}
\end{proof}

Our second corollary
shows that all tuples $(r,p) \in \{(0,0)\} \cup \N^2$ can be realized as $(\reg(\K[G]), \pdim(\K[G]))$ for some connected bipartite graph $G$.
Here
$\mathbb{N} = \{1,2,3,\ldots,\}$.

\begin{corollary} \label{union}
    Let $r$ and $p$ be integers. Then there is a connected bipartite graph $G$ on at least two vertices with $\reg(\K[G])=r$ and $\pdim(\K[G])=p$ if and only if $r=p=0$ or $r,p \geq 1$. Equivalently, 
    \[ \bigcup_{n\geq 2}^\infty \textnormal{CBPT}^\textnormal{pdim}_\textnormal{reg} (n) = \{(0,0)\} \cup \N^2.\]
\end{corollary}
\begin{proof}
    Let $G$ be a connected bipartite graph on
    $n \geq 2$ vertices with  $\reg(\K[G])=r$ and $\pdim(\K[G])=p$.   By Theorem \ref{reg vs pdim},
    $(r,p) = (0,0)$, or $r,p \geq 1$.

    Conversely, suppose $r=p=0$ or $r,p \geq 1$. If $r=p=0$, then the unique tree $G$ on two vertices
    is a connected bipartite graph with 
    $\reg(\K[G]) = \pdim(\K[G]) = 0$,
    so $(0,0) \in {\rm CBPT}_{\reg}^{\pdim}(2)$. So assume $r,p \geq 1$. Let 
    $N =2+r+ \max\{r,p\}$. Then $N \geq 2+r+r=2+2r$. Hence 
    \[
    0< r < r+1 = \left \lfloor{\frac{2r+2}{2}}\right \rfloor \leq \left \lfloor{\frac{N}{2}}\right \rfloor
    \] 
    and thus $0 < r < \left \lfloor{\frac{N}{2}}\right \rfloor$. Also, observe that since $r \geq 1$,
    \[r(N-2-r)=r\max\{r,p\} \geq rp \geq p,\] so $1 \leq p \leq r(N-2-r)$. 
    We then have $(r,p) \in \textnormal{CBPT}^\textnormal{pdim}_\textnormal{reg} (N)$
    by Theorem \ref{reg vs pdim}.
\end{proof}

\section{Proof of main theorem}

Using the previous sections, we can prove the main result
of this paper, namely, a description of all the elements
of ${\rm CBPT}_{{\rm reg},\deg, {\rm pdim},{\rm depth},\dim}(n)$.
In particular, we now prove:

\begin{theorem}\label{resultmaintheorem}
Let $n \geq 2$.  Then ${\rm CBPT}_{{\rm reg},\deg, {\rm pdim},{\rm depth},\dim}(n)$ is given by
\[\left\{(r,r,p,n-1,n-1) ~\left|~ 
0 < r < \left\lfloor \frac{n}{2} \right\rfloor,~
1 \leq p \leq r(n-2-r) \right\}\right. \cup \{(0,0,0,n-1,n-1)\}.\]
\end{theorem}

\begin{proof}
    Fix an $n \geq 2$ and let $T$ denote the
    set in the statement.  We will first show that
    all the elements of ${\rm CBPT}_{{\rm reg},\deg, {\rm pdim},{\rm depth},\dim}(n)$ belong to $T$.

Let $G$ be any connected bipartite graph on $n$ vertices,
and set 
\[(r,d_1,p,d_2,d_3) = (\reg(\K[G],\deg(h_{\K[G]}(t)),\pdim(\K[G]),
\depth(\K[G]),\dim(\K[G])).\]
By Corollary \ref{p=q-n+1} $(i)$ and $(ii)$, we 
have $0 \leq r=d_1 < \left\lfloor \frac{n}{2} \right\rfloor$ and $d_2=d_3=n-1$, i.e., $(r,d_1,p,d_2,d_3) = (r,r,p,n-1,n-1)$.
Since $G \in {\rm CBPT}(n)$, by 
Theorem \ref{reg vs pdim}
we have $r = p =0$, or $0 < r < \left\lfloor \frac{n}{2} \right\rfloor$
and $1 \leq p \leq r(n-2-r)$.  Consequently, $(r,d_1,p,d_2,d_3)
\in T$.

For the reverse containment, note that 
$(0,0,0,n-1,n-1) \in {\rm CBPT}_{{\rm reg},\deg, {\rm pdim},{\rm depth},\dim}(n)$ since any connected tree $G$ on $n$ vertices
satisfies 
\[(\reg(\K[G],\deg(h_{\K[G]}(t)),\pdim(\K[G]),
\depth(\K[G]),\dim(\K[G])) = (0,0,0,n-1,n-1)\]
by Corollary \ref{p=q-n+1} and Lemma \ref{reg of tree}.
So, consider any $(r,r,p,n-1,n-1) \in T$ with $0 <r$.  
Because the tuple $(r,p)$ belongs to 
${\rm CPBT}_{\reg}^{\pdim}(n)$ by Theorem \ref{reg vs pdim},
there exists a connected bipartite graph $G$ on $n$ 
vertices with $\reg(\K[G]) = r$ and $\pdim(\K[G]) = p$.  But
by Corollary \ref{p=q-n+1} $(i)$ and $(ii)$, this graph
$G$ also has $\deg(h_{\K[G]}(t)) = r$ and $\dim(\K[G]) =
\depth(\K[G]) = n-1$.  Thus $(r,r,p,n-1,n-1) \in
{\rm CBPT}_{{\rm reg},\deg, {\rm pdim},{\rm depth},\dim}(n)$,
as desired.
\end{proof}

\begin{remark}
    Theorem \ref{resultmaintheorem} focuses on the {\it connected} 
    bipartite graphs.  It is possible to provide a generalization
    of Theorem \ref{resultmaintheorem} to describe 
    all the possible values for 
    these invariants
    for all bipartite graphs,
    not just connected bipartite graphs.  In particular,
    one needs to make use of the fact that these invariants
    behave well over tensor products.
    However, additional
    care needs to be taken for bipartite graphs with
    isolated vertices.  See \cite{B} for the worked out details.
\end{remark}

\begin{remark}\label{codegree}
    In Theorem \ref{reg vs pdim}, we 
    completely determined the set
    ${\rm CBPT}_{\reg}^{\pdim}(n)$.   Akihiro Higashitani pointed out to us that describing
    this set is equivalent to describing a combinatorially defined set; we quickly sketch out these details.
    Associated with a toric ideal of a graph $G$ is a polytope $P_G$.  The {\it codegree} of $P_G$
    is given by ${\rm codeg}(P_G) = \min\{ k \in \mathbb{Z} ~|~ {\rm int}(kP_G) \cap \mathbb{Z}^n \neq 
    \emptyset\}$.  This invariant is measuring
    the smallest integer $k$ such that the interior of the polytope $kP_G$ has an integer lattice point.
    In the case $G$ is a bipartite graph, it can
    be shown that 
    ${\rm deg}(h_{\K[G]}(t)) + {\rm codeg}(P_G) = n$.
    Further, when $G$ is a bipartite graph, by
    Corollary \ref{p=q-n+1} $(ii)$ and $(iii)$
    we have ${\rm codeg}(P_G) = n- {\rm reg}(\K[G])$
    and $|E| = {\rm pdim}(\K[G])+n-1$.  
    Consequently,
    determining the elements of  ${\rm CBPT}_{\reg}^{\pdim}(n)$ is equivalent
    to determining the elements of 
    $${\rm CBPT}_{{\rm codeg}}^{|E|}(n)
    = \{({\rm codeg}(P_G),|E|) ~|~ 
    G \in {\rm CBPT}(n)\}.$$
    This observation suggests it might be interesting
    to consider all pairs $({\rm codeg}(P_G),|E|)$
    for all graphs, not just bipartite graphs.
  
\end{remark}


\bibliographystyle{plain}
\bibliography{main}

\begin{thebibliography}{10}

\bibitem{ADS}
Ayah Almousa, Anton Dochtermann, and Ben Smith.
\newblock Root polytopes, tropical types, and toric edge ideals.
\newblock Preprint, arXiv:2209.09851, 2022.

\bibitem{B}
Kieran Bhaskara.
\newblock Comparing invariants of toric ideals of bipartite graphs.
\newblock Master's thesis, McMaster University, 2023.

\bibitem{BOVT}
Jennifer Biermann, Augustine O'Keefe, and Adam Van~Tuyl.
\newblock Bounds on the regularity of toric ideals of graphs.
\newblock {\em Adv. in Appl. Math.}, 85:84--102, 2017.

\bibitem{BH}
Winfried Bruns and J\"{u}rgen Herzog.
\newblock {\em Cohen--{M}acaulay rings}, volume~39 of {\em Cambridge Studies in
  Advanced Mathematics}.
\newblock Cambridge University Press, Cambridge, 1993.

\bibitem{EH}
Nursel Erey and Takayuki Hibi.
\newblock The size of {B}etti tables of edge ideals arising from bipartite
  graphs.
\newblock {\em Proc. Amer. Math. Soc.}, 150(12):5073--5083, 2022.

\bibitem{FKVT}
Giuseppe Favacchio, Graham Keiper, and Adam Van~Tuyl.
\newblock Regularity and {$h$}-polynomials of toric ideals of graphs.
\newblock {\em Proc. Amer. Math. Soc.}, 148(11):4665--4677, 2020.

\bibitem{FS}
Antonino Ficarra and Emanuele Sgroi.
\newblock The size of the \uppercase{B}etti table of binomial edge ideals.
\newblock Preprint, arXiv:2302.03585, 2023.

\bibitem{M2}
Daniel~R. Grayson and Michael~E. Stillman.
\newblock Macaulay2, a software system for research in algebraic geometry.
\newblock Available at \url{https://faculty.math.illinois.edu/Macaulay2/}.

\bibitem{HaH}
Huy~T\`ai H\`a and Takayuki Hibi.
\newblock M{AX} {MIN} vertex cover and the size of {B}etti tables.
\newblock {\em Ann. Comb.}, 25(1):115--132, 2021.

\bibitem{HH}
J\"{u}rgen Herzog and Takayuki Hibi.
\newblock The regularity of edge rings and matching numbers.
\newblock {\em Mathematics}, 8(1):103--114, 2020.

\bibitem{HHO}
J\"{u}rgen Herzog, Takayuki Hibi, and Hidefumi Ohsugi.
\newblock {\em Binomial ideals}, volume 279 of {\em Graduate Texts in
  Mathematics}.
\newblock Springer, Cham, 2018.

\bibitem{HHKO}
Takayuki Hibi, Akihiro Higashitani, Kyouko Kimura, and Augustine~B. O'Keefe.
\newblock Depth of edge rings arising from finite graphs.
\newblock {\em Proc. Amer. Math. Soc.}, 139(11):3807--3813, 2011.

\bibitem{HKKMV}
Takayuki Hibi, Hiroju Kanno, Kyouko Kimura, Kazunori Matsuda, and Adam
  Van~Tuyl.
\newblock Homological invariants of {C}ameron--{W}alker graphs.
\newblock {\em Trans. Amer. Math. Soc.}, 374(9):6559--6582, 2021.

\bibitem{HKMVT}
Takayuki Hibi, Kyouko Kimura, Kazunori Matsuda, and Adam Van~Tuyl.
\newblock The regularity and {$h$}-polynomial of {C}ameron--{W}alker graphs.
\newblock {\em Enumer. Comb. Appl.}, 2(3):Paper No. S2R17, 12, 2022.

\bibitem{HM1}
Takayuki Hibi and Kazunori Matsuda.
\newblock Regularity and {$h$}-polynomials of monomial ideals.
\newblock {\em Math. Nachr.}, 291(16):2427--2434, 2018.

\bibitem{HM2}
Takayuki Hibi and Kazunori Matsuda.
\newblock Lexsegment ideals and their {$h$}-polynomials.
\newblock {\em Acta Math. Vietnam.}, 44(1):83--86, 2019.

\bibitem{HM3}
Takayuki Hibi and Kazunori Matsuda.
\newblock Regularity and {$h$}-polynomials of binomial edge ideals.
\newblock {\em Acta Math. Vietnam.}, 47(1):369--374, 2022.

\bibitem{HMVT}
Takayuki Hibi, Kazunori Matsuda, and Adam Van~Tuyl.
\newblock Regularity and {$h$}-polynomials of edge ideals.
\newblock {\em Electron. J. Combin.}, 26(1):Paper 1.22, 11, 2019.

\bibitem{HKO}
Akihiro Higashitani, Akane Kanno, and Ryota Ueji.
\newblock Behaviors of pairs of dimensions and depths of edge ideals.
\newblock Preprint, arXiv:2301.04763, 2023.

\bibitem{J}
Bill Jackson.
\newblock Long cycles in bipartite graphs.
\newblock {\em J. Combin. Theory Ser. B}, 38(2):118--131, 1985.

\bibitem{KK}
Thomas Kahle and Jonas Kr\"{u}semann.
\newblock Binomial edge ideals of cographs.
\newblock {\em Rev. Un. Mat. Argentina}, 63(2):305--316, 2022.

\bibitem{VV}
Carlos~E. Valencia and Rafael~H. Villarreal.
\newblock Canonical modules of certain edge subrings.
\newblock {\em European J. Combin.}, 24(5):471--487, 2003.

\bibitem{Vas}
Wolmer~V. Vasconcelos.
\newblock {\em Computational methods in commutative algebra and algebraic
  geometry}, volume~2 of {\em Algorithms and Computation in Mathematics}.
\newblock Springer-Verlag, Berlin, 1998.
\newblock With chapters by David Eisenbud, Daniel R. Grayson, J\"{u}rgen Herzog
  and Michael Stillman.

\bibitem{Villarreal-article}
Rafael~H. Villarreal.
\newblock Rees algebras of edge ideals.
\newblock {\em Comm. Algebra}, 23(9):3513--3524, 1995.

\bibitem{Villarreal-book}
Rafael~H. Villarreal.
\newblock {\em Monomial algebras}.
\newblock Monographs and Research Notes in Mathematics. CRC Press, Boca Raton,
  FL, second edition, 2015.

\end{thebibliography}
\end{document}